\documentclass[11pt]{article}

\usepackage{amsfonts,amssymb}
\usepackage{latexsym,times}
\usepackage{amsmath}
\usepackage{amscd}
\usepackage{epsfig,color}

\newcommand{\K}{\mathbb{K}}

\newtheorem{theorem}{Theorem}[section]

\newtheorem{proposition}[theorem]{Proposition}

\newtheorem{definition}[theorem]{Definition}

\newtheorem{example}[theorem]{Example}

\newtheorem{remark}[theorem]{Remark}

\newtheorem{proof}{Proof}

\input{epsf.sty}

\begin{document}

\title{Deformations of  Hom-Alternative and Hom-Malcev algebras }

\author{Mohamed Elhamdadi
 \footnote{ Email: \texttt{emohamed@math.usf.edu}}\\
 University of South Florida
 \and Abdenacer Makhlouf
 \footnote{Email: \texttt{Abdenacer.Makhlouf@uha.fr}}\\
 Universit\'e de Haute Alsace}

\date{}

\maketitle

\begin{abstract}
The aim of this paper  is to extend Gerstenhaber formal deformations of algebras to the case of Hom-Alternative and Hom-Malcev algebras.   We construct deformation cohomology groups in low dimensions.
 Using a composition construction,  we  give a procedure to provide deformations of alternative algebras  (resp.  Malcev algebras)  into Hom-alternative algebras  (resp.  Hom-Malcev algebras). Then it is used to supply examples for which we compute some cohomology invariants.
\end{abstract}
\noindent {\small{\emph{\textbf{Keyword}.
Hom-alternative algebra, Hom-Malcev algebra, formal deformation, cohomology.}}}

\section*{Introduction}

The Hom-Lie algebras were introduced in \cite{HLS,LS} to study the deformations of Witt algebra, which is the complex Lie algebra of derivations of the Laurent polynomials in one variable, and the deformation of the Virasoro algebra, a one-dimensional central extension of Witt algebra.   Since then Hom-structures in different settings (algebras, coalgebras, Hopf algebras, Leibniz algebras, $n$-ary algebras etc) were investigated by many authors (see for example \cite{F.Ammar,AMS,ArMS,AM2008,Canepl2009,FregierGohr1,HLS,LS,Makhlouf11,MS,MakhloufSilv2,MS4,Yau1,Yau3}).  The Hom-alternative algebras were introduced by the second author in \cite{Makhlouf10}, while Hom-Malcev algebras were introduced by D. Yau in \cite{Yau4}, where connections between Hom-alternative algebras and Hom-Malcev algebras are given.\\
A deformation of a mathematical object (for example, analytic, geometric or algebraic structures) is a family of the same kind of objects which depend on some parameter.  One usually asks does there exist a parameter family of similar structures, such that for an initial value of the parameter one gets the same structure one started with.\\
In the fifties,  Kodaira and Spencer developed a more or less systematic theory of deformations of complex structures of higher dimensional manifolds. The concept of deformations of complex structures, or of a family of complex structures depending differentiably on a parameter, can be defined in terms of structure tensors determining the complex structures.  Soon in the sixties, Gerstenhaber generalized this to the context of algebraic and homological setting \cite{Gerst1,Gerst2,Gerst3,Gerst4}.  He gave a unified treatment of the subjects of deformations of algebras  and the cohomology modules on which the analysis of deformations depends. See also \cite{CCES1,CCES2,Fialowski86,LaudalLNM,MakhloufDeform07,Makhlouf-Hopf,MStash} for deformation theory.\\
The deformations of Hom-algebras were initiated by the second author and S. Silvestrov in \cite{MakhloufSilv}, where the deformations of Hom-associative and Hom-Lie algebras were investigated.  Further developments on the cohomology of Hom-associative and Hom-Lie algebras were given in \cite{AEM}, where cochain complexes were provided.  Independently the cochain complex in the case of Hom-Lie algebras was given in \cite{YS}.   \\
In this paper we extend the theory of formal deformation \`a la M. Gerstenhaber to the context of Hom-alternative and Hom-Malcev algebras.  This can be seen as an extension of the work done by the authors in \cite{ElhamMakh1}. Using the general procedure of opposite multiplication, one turns left Hom-alternative algebras into right Hom-alternative algebras and vice-versa. Thus we will restrict ourselves to the case of left Hom-alternative algebras since  any statement with left Hom-alternative algebra has its corresponding statement for right Hom-alternative algebra.

The paper is organized as follows. In Section 1, we review the basic definitions and properties of  Hom-alternative algebras, Hom-Malcev algebras and give examples.  In Section 2, we establish the formal deformation theory of Hom-alternative algebras  and also give some elements of a cohomology of these algebras.  In Section 3,  we provide a way to obtain formal deformations of alternative algebras into Hom-alternative algebras using a composition process.  Section 4 is dedicated to supplying examples and computations of derivations and cocycles of 4-dimensional Hom-alternative algebras. We conclude the paper by Section 5 in which we study, in a similar way,  formal deformations of  Malcev and Hom-Malcev algebras. Some computations in this paper were done using a software system of computer algebra.

\section{Preliminaries}
We start by recalling the notions of Hom-alternative and Hom-Malcev  algebras.  Throughout this paper $\mathbb{K}$ is an algebraically closed field of
characteristic zero and $A$ is a vector space over $\mathbb{K}$. We mean by a Hom-algebra a triple $(A,\mu,\alpha )$ consisting of a vector space $A$, a bilinear map $\mu$ and  a linear map $\alpha$. In all the examples involving multiplication, the unspecified products are either given by skewsymmetry, when the algebra is skewsymmetric,  or equal to zero.

First we recall the definition and some properties of Hom-alternative algebras.

\begin{definition} [\cite{Makhlouf10}] {\rm
 A \emph{left Hom-alternative} algebra (resp. \emph{right
Hom-alternative algebra}) is a triple $(A,\mu,\alpha )$
consisting of a $\K$-vector space $A$, a multiplication $\mu: A\otimes A \rightarrow A$ and a linear map $\alpha: A
\rightarrow A$
satisfying the left Hom-alternative identity, that is for any $x,y$ in $A$,
\begin{equation}\label{HomLeftAlternative}
\mu (\alpha (x),\mu(x, y))=\mu(\mu(x, x) , \alpha (y)),
\end{equation}
respectively, right Hom-alternative identity, that is
\begin{equation}\label{HomRightAlternative}
\mu (\alpha(x),\mu(y, y))=\mu(\mu(x,y) ,\alpha( y)).
\end{equation}

A Hom-alternative algebra is one which is both left and right
Hom-alternative algebra.

}
\end{definition}

\begin{definition} {\rm
Let $\left( A,\mu ,\alpha \right) $ and $\left( A^{\prime },\mu
^{\prime },\alpha^{\prime }\right) $ be two Hom-alternative
algebras. A linear map $f\ :A\rightarrow A^{\prime }$ is said to be
a
\emph{morphism of Hom-alternative algebras} if the following holds
$$  \mu ^{\prime }\circ (f\otimes f)=f\circ \mu\quad
\text{ and } \qquad f\circ \alpha=\alpha^{\prime }\circ f .
$$

}
\end{definition}

\begin{remark} [\cite{MS}]{\rm
Notice that Hom-associative algebras are Hom-alternative algebras. Since a Hom-associative algebra is a Hom-algebra $( A, \mu,
\alpha) $  satisfying
$
\mu(\alpha(x), \mu (y, z))= \mu (\mu (x, y), \alpha (z)), \;\text{for any} \;$x,y,z$ \;\text{in}\; $A$.
$
}
\end{remark}
The following is  an example of Hom-associative algebra which of course is an alternative algebra.
\begin{example}\label{example1ass}
Let $A=\K^3$ be the $3$-dimensional vector space
 over $\K$ with a basis $\{e_1,e_2,e_3\}$. The following multiplication $\mu$ and linear map
$\alpha$   on $A$ define Hom-associative algebras.
$$
\begin{array}{ll}
\begin{array}{lll}
 \mu ( e_1,e_1)&=& a\ e_1, \ \\
\mu ( e_1,e_2)&=&\mu ( e_2,e_1)=a\ e_2,\\
\mu ( e_1,e_3)&=&\mu ( e_3,e_1)=b\ e_3,\\
 \end{array}
 & \quad
 \begin{array}{lll}
\mu ( e_2,e_2)&=& a\ e_2, \ \\
\mu ( e_2, e_3)&=& b\ e_3, \ \\
\
  \end{array}
\end{array}
$$and
$$  \alpha (e_1)= a\ e_1, \quad
 \alpha (e_2) =a\ e_2 , \quad
   \alpha (e_3)=b\ e_3,
$$
 where $a,b$ are parameters in $\K$.

 \noindent
 When $a\neq b$ and $b\neq 0$, the equality
$\mu (\mu (e_1,e_1),e_3))- \mu ( e_1,\mu
(e_1,e_3))=(a-b)b e_3$ makes this algebra  neither associative nor left alternative.

\end{example}
Now, we give another example of left Hom-alternative algebra. More general examples will be supplied further in the paper.
\begin{example}
Let $A=\K^4$ be the  $4$-dimensional vector space
 over $\K$ with a basis $\{e_0,e_1,e_2,e_3\}$. The following multiplication $\mu$ and linear map
$\alpha$ on $A$ define  a left Hom-alternative algebra.
$$
\mu(e_0,e_0)=e_0+e_2,\quad
\mu (e_2,e_0)=2 \ e_2,\quad
\mu (e_3,e_0)= \ e_2,
$$
$$
 \alpha(e_0)= e_0+\ e_2,\ \
\alpha(e_1)=0,\ \
\alpha(e_2)=2 \ e_2,\ \
\alpha(e_3)= \ e_2,
$$
This Hom-alternative algebra is not alternative since $$\mu(\mu(e_0+e_3,e_0+e_3),e_0)-\mu(e_0+e_3,\mu(e_0+e_3,e_0))=-e_2\neq 0.$$

\end{example}
\begin{remark}
The  \emph{Hom-associator} of a
 Hom-algebra $( A, \mu,
\alpha) $ is a trilinear map denoted by $\mathfrak{as_\alpha}$,
and   defined for any $x,y,z\in A$ by
\begin{equation}\label{HomAssociator}
\mathfrak{as}_\alpha(x,y,z)=\mu (\alpha(x),\mu (y, z))-\mu(\mu(x,
y), \alpha(z)).
\end{equation}
In terms of Hom-associator, the identities \eqref{HomLeftAlternative} and
\eqref{HomRightAlternative} may be written:
$
 \mathfrak{as}_\alpha(x,x,y)=0$ and $\mathfrak{as}_\alpha(y,x,x)=0.
$\\
They are also equivalent, by linearization, to

\begin{equation}\label{HomLeftAlternativeLineariz}
\mu (\alpha (x),\mu(y, z))-\mu(\mu(x, y),\alpha( z))+\mu (\alpha
(y), \mu(x, z))-\mu(\mu(y, x), \alpha (z))=0.
\end{equation}
respectively,
\begin{equation}\label{HomRightAlternativeLineariz}
\mu (\alpha(x),\mu (y, z))-\mu (\mu (x, y), \alpha(z))+\mu
(\alpha(x), \mu (z, y))-\mu (\mu (x, z), \alpha(y))=0.
\end{equation}
\end{remark}

  \begin{remark} The multiplication could be considered as a linear map
  $\mu : A \otimes A \rightarrow A$, then the condition
  \eqref{HomLeftAlternativeLineariz} and  \eqref{HomRightAlternativeLineariz} can be written
  \begin{equation}\label{HomLeftAlternativeLineariz2}
\mu \circ (\alpha\otimes \mu-\mu \otimes \alpha)\circ (id^{\otimes
3} +\sigma_{1})=0,
\end{equation}
respectively
\begin{equation}\label{HomRightAlternativeLineariz2}
\mu \circ (\alpha\otimes \mu-\mu \otimes \alpha)\circ (id^{\otimes
3}+\sigma_{2})=0.
\end{equation}
where $id$ stands for the identity map and $\sigma_{1}$ and
$\sigma_{2}$ stand for trilinear maps defined for any $x,y, z\in A$ by
$
\sigma_{1}(x\otimes y\otimes z)=y\otimes x\otimes z$ and $
\sigma_{2}(x\otimes y\otimes z)=x\otimes z\otimes
y.
$

\noindent Therefore, the identities
\eqref{HomLeftAlternativeLineariz} and
\eqref{HomRightAlternativeLineariz} are equivalent respectively to
\begin{equation}\label{Alt0}\mathfrak{as}_\alpha+\mathfrak{as}_\alpha\circ\sigma_{1}=0 \quad
\text{ and  }\
\mathfrak{as}_\alpha+\mathfrak{as}_\alpha\circ\sigma_{2}=0.
\end{equation}
  \end{remark}
  Hence, for any $x,y,z\in A$, we have
  \begin{equation}\label{Alt1}
  \mathfrak{as}_\alpha(x,y,z)=-\mathfrak{as}_\alpha(y,x,z)\quad \text{ and
  }\quad \mathfrak{as}_\alpha(x,y,z)=-\mathfrak{as}_\alpha(x,z,y).
  \end{equation}

\noindent
The identity  \eqref{Alt1} leads to the following characterization of Hom-alternative algebras
\begin{proposition}
A triple  $(A,\mu,\alpha )$ defines a   Hom-alternative algebra if and only if the Hom-associator
$\mathfrak{as}_\alpha(x,y,z)$ is an alternating function of its
arguments, that is
$$\mathfrak{as}_\alpha(x,y,z)=-\mathfrak{as}_\alpha(y,x,z)=
-\mathfrak{as}_\alpha(x,z,y)=-\mathfrak{as}_\alpha(z,y,x).$$
\end{proposition}
The proof can be found in \cite{Makhlouf10}. Further properties of Hom-alternative algebras could be found in \cite{Makhlouf10} and \cite{Yau4}.
\noindent
Now, we give  the definition of Hom-Malcev algebras and their connection to Hom-alternative algebras.
\begin{definition}[\cite{Yau4}] \label{def:HomMalcev}
A \emph{Hom-Malcev algebra} is a triple $(A, [\ , \ ], \alpha)$
where $[\ , \ ]: A\times A \rightarrow A$ is a skewsymmetric bilinear map
and $\alpha: A \rightarrow A$
 a linear map
 satisfying for all $x, y, z\in A$
\begin{equation}\label{HomMalcevIdentity} J_\alpha(\alpha(x),\alpha(y),[x,z])=[J_\alpha(x,y,z),\alpha^2(x)]  \quad
{\text{(Hom-Malcev identity)}}
\end{equation}
 where $J_\alpha$ is the Hom-Jacobiator which is a trilinear map defined by $J_\alpha(x,y,z)=\circlearrowleft_{x,y,z}{[[x,y],\alpha(z),]}$  with $\circlearrowleft_{x,y,z}$ denoting the
summation over the cyclic permutation on $x,y,z$.
\end{definition}

\noindent Likewise when $\alpha$ is the identity the Hom-Malcev identity reduces to classical Malcev identity which is equivalent, using skewsymmetry to
$$[[x,y],[x,z]]=[[[x,y],z],x]+[[[y,z],x],x]+[[[z,x],x],y]
$$

\begin{remark}
Recall that a Hom-Lie algebra   is  a Hom-algebra  $(A, [\ , \ ], \alpha)$
where the bracket is skewsymmetric and satisfies, for $x,y,z\in A$,
$J_\alpha(x,y,z)=0.$ Then any Hom-Lie algebra is a Hom-Malcev algebra.
\end{remark}
\begin{example}[\cite{Yau4}]\label{MalcevExample}
Let $A=\K^4$ be the  $4$-dimensional vector space
 over $\K$ with a basis $\{e_0,e_1,e_2,e_3\}$. The following bracket $[-,-]$ and linear map
$\alpha$ on $A$ define  a Hom-Malcev algebra.
\begin{align*}
& [e_0,e_1]=-(b_1e_1+b_2e_2+a_2b_1e_3), \quad
&[e_0,e_2]=-c e_2, \quad
& [e_0,e_3]=b_1 c \ e_3,\\
& [e_1,e_0]=b_1e_1+b_2e_2+a_2b_1e_3, \quad
& [e_1,e_2]=2b_1 c \ e_3, \quad
& [e_2,e_0]=c e_2,\\
& [e_2,e_1]=-2b_1 ce_3,\quad
& [e_3,e_0]=-b_1 ce_3.\quad &&  \
\end{align*} and
\begin{align*}
 \alpha_1(e_0)&= e_0+a_2\ e_2+a_3e_3,\quad &
\alpha_1(e_1)&=b_1e_1+b_2e_2+a_2b_1e_3,\\
\alpha_1(e_2)&=c \ e_2,\quad &
\alpha_1(e_3)&= b_1c\ e_3,
\end{align*}

\end{example}
\noindent
We show in the following the connection  between Hom-alternative algebras and Hom-Malcev  algebras given in \cite{Yau4}. The Hom-alternative algebras are related to Hom-Malcev algebras as Hom-associative algebras to Hom-Lie algebras (see \cite{MS})
\begin{theorem}[\cite{Yau4}]
Let $(A,\mu,\alpha )$ be a Hom-alternative algebra.  Then $(A,[ \ ,\  ],\alpha )$, where the bracket is defined for all $x,y \in A$ by
$$
[ x,y ]=\mu (x,y)-\mu (y,x )
$$
is a  Hom-Malcev algebra.
\end{theorem}
\noindent
We refer to \cite{Yau4} for more properties on Hom-Malcev algebras and Hom-Malcev-admissible algebras.

\noindent
\section{Formal Deformations of  Hom-Alternative algebras}
We develop, in this section, a deformation theory for Hom-Alternative algebras by analogy with Gerstenhaber deformations (\cite{Gerst1}, \cite{Gerst2}, \cite{Gerst3}, \cite{Gerst4}).
Heuristically, a formal deformation of an algebra $A$ is $1$-parameter family of multiplication (of the same sort) obtained by perturbing the multiplication of $A$.
\noindent
Let  $({A },\mu _{0}, \alpha_0)$ be a left Hom-alternative
algebra. Let $\K [[t]]$ be the power series ring in one
variable $t$ and coefficients in $\K $ and let $A[[t]]$ be the set of
formal power series whose coefficients are elements of $A$
(note that $A[[t]]$ is obtained   by extending the coefficients domain of $A$
from $\K $ to $\K [[ t]]$). Then $A[[t]]$ is a $\K[[t]]$-module.
When $A$ is finite-dimensional, we have $ A[[t]]=A\otimes _{\K
}\K[[t]]$. One notes that $A$ is a submodule of $A[[t]]$. Given a
 $\K$-bilinear map $f :A\times A  \rightarrow A$, it admits naturally an
 extension to a $\K[[t]]$-bilinear map
 $f :A[[t]]\otimes A[[t]]  \rightarrow A[[t]]$, that is,
 if  $x=\sum_{i\geq0}{a_i t^i}$ and $y=\sum_{j\geq0}{b_j t^j}$ then
$f(x\otimes y)=\sum_{i\geq0,j\geq0}{t^{i+j}f (a_i\otimes b_j)}$.

\begin{definition}{\rm
Let  $({A },\mu _{0}, \alpha_0)$ be a left Hom-alternative
algebra. A \emph{formal  left Hom-alternative deformation} of
${A }$  is given by the $\K[[t]]$-bilinear
 map $\mu_{t} :A[[t]]\otimes A[[t]]  \rightarrow
A[[t]]$  and the $\K[[t]]$-linear
 map
$\alpha_{t} :A[[t]]  \rightarrow
A[[t]]$,
such that
$\mu_{t} =\sum_{i\geq 0}\mu_{i}t^{i},
$
where each $\mu_{i}$ is a $\K$-bilinear map $\mu_{i}:  A\otimes A
\rightarrow  A$ (extended to be $\K[[t]]$-bilinear),  and $\alpha_{t} =\sum_{i\geq 0}\alpha_{i}t^{i},
$ where each $\alpha_{i}$ is a $\K$-linear map $\mu_{i}:  A
\rightarrow  A$ (extended to be $\K[[t]]$-linear)
such that
for $x, y,z\in
 A$, the following
formal  left Hom-alternativity identity holds
\begin{equation}\label{equ1}
\mu_{t}(\alpha_{t}(x) \otimes \mu_{t}(y\otimes z))-\mu_{t}(\mu_{t}(x\otimes y)\otimes \alpha_{t}(z)))+\mu _t\left( \alpha_{t}(y)\otimes \mu _t\left( x\otimes z\right) \right)-\mu _t\left( \mu _t\left( y\otimes x\right)
\otimes \alpha_{t}(z)\right) =0.
\end{equation}
}
\end{definition}
\noindent
The identity \eqref{equ1} is called the deformation equation of the Hom-alternative algebras. Notice that here both the multiplication and the linear map are deformed.

\subsection{Deformations equation}
Now we investigate the deformation equation.  We give conditions on $\mu _i$ in order for the deformation $\mu _t$ to be
Hom-alternative. Expanding the left hand side of the equation (\ref{equ1}) and collecting the coefficients of $%
t^k$ yields an infinite system of equations given, for any nonnegative integer  $ \
   k, $  by
\begin{eqnarray}\label{equaDefogeneral}
\lefteqn{
\sum_{i=0}^k\sum_{j=0}^{k-i}{\mu _i} ( \alpha_j(x) \otimes \mu _{k-i} (
y\otimes z ))-\mu _i ( \mu
_{k-i} ( x\otimes y ) \otimes \alpha_j(z) ) +} \nonumber \\
&&
 \left.
 \mu _i ( \alpha_j(y) \otimes \mu _{k-i} (
 x\otimes z ) )
  -\mu _i ( \mu
_{k-i} ( y\otimes x ) \otimes \alpha_j(z) )=0. \right.
\end{eqnarray}
\noindent
The first equation corresponding to $k=0$, is the left Hom-alternativity identity of  $({A },\mu _{0}, \alpha_0)$. \\The second equation    corresponding to  $k=1$ can be written
\begin{eqnarray}\label{delta2general}
[\mu_0\circ(\mu _1 \otimes {\rm \alpha_0} - {\rm  \alpha_0} \otimes \mu _1)+\mu _1 \circ(\mu_0 \otimes {\rm  \alpha_0}-{\rm  \alpha_0} \otimes \mu_0)\\
+\mu _0 \circ(\mu_0 \otimes {\rm  \alpha_1}-{\rm  \alpha_1} \otimes \mu_0)] \circ ({\rm id}^{\otimes 3}+\sigma_1)=0,
\end{eqnarray}
where $\sigma_1$ is defined on $A^{\otimes 3}$ by   $\sigma_1(x\otimes y \otimes z)=y \otimes x \otimes z.$

\subsection{Elements of Cohomology}
We provide some elements of a cohomology theory motivated by  formal deformation theory in  the case when  $\alpha_0$ is not deformed. Then the deformation equation becomes
\begin{eqnarray}\label{equaDefo}
\lefteqn{
\sum_{i=0}^k{\mu _i} ( \alpha(x) \otimes \mu _{k-i} (
y\otimes z ))-\mu _i ( \mu
_{k-i} ( x\otimes y ) \otimes \alpha(z) ) +} \nonumber \\
&&
 \left.
 \mu _i ( \alpha(y) \otimes \mu _{k-i} (
 x\otimes z ) )
  -\mu _i ( \mu
_{k-i} ( y\otimes x ) \otimes \alpha(z) )=0. \right.
\end{eqnarray}
Therefore, the equation \eqref{delta2general} is reduced to
\begin{equation}\label{delta2}
[\mu_0\circ(\mu _1 \otimes {\rm \alpha_0} - {\rm  \alpha_0} \otimes \mu _1)+\mu _1 \circ(\mu_0 \otimes {\rm  \alpha_0}-{\rm  \alpha_0} \otimes \mu_0)] \circ ({\rm id}^{\otimes 3}+\sigma_1)=0,
\end{equation}
which suggests that $\mu _1$ should be a 2-cocycle for a certain left
Hom-alternative algebra cohomology. In the sequel we define first and second  coboundary operators fitting with deformation theory.
\noindent
Let  $(A,\mu,\alpha)$ be a Hom-alternative algebra and  let $ \mathcal{C}^1 ( A,
A )$ be the set of linear maps  $f: A\rightarrow A$  which commute with $\alpha$, and $ \mathcal{C}^2 ( A,
A )$ be the set of bilinear maps on $A$.  We  define the first differential $\delta ^1f \in  \mathcal{C}^2 ( A,
A )$ by
\begin{equation}
\delta ^1f =\mu \circ \left(
f\otimes  {\rm id} \right) +\mu \circ \left( {\rm id}  \otimes f \right)
-f\circ \mu.
\end{equation}
\noindent
The map  $f$ is said to be a $1$-cocycle if $\delta ^1f=0$.
We remark that the first differential of a left Hom-alternative algebra is similar to the first differential  map of Hochschild cohomology of associative algebras. The $1$-cocycles are derivations.
\noindent
Let $\phi\in\mathcal{C}^2 ( A,
A )$, we define the second differential $\delta^2 \phi  \in \mathcal{C}^3 ( A,
A )$ by
 \begin{equation}\label{delta2bis}
\delta^2\phi=[\mu\circ(\phi \otimes {\alpha} - {\alpha} \otimes \phi)+\phi\circ(\mu \otimes {\alpha}-{\alpha} \otimes \mu)] \circ ({\rm id}^{\otimes 3}+\sigma_1).
\end{equation}
where $\sigma_1$ is defined on $A^{\otimes 3}$ by   $\sigma_1(x\otimes y \otimes z)=y \otimes x \otimes z$.

\begin{proposition}{\it
The composite $\delta ^{2}\circ \delta ^{1}$ is zero.}
\end{proposition}
{\it Proof.}
Let $x, y, z \in A$ and  $f\in \mathcal{C}^{1}({A ,A })$,
\begin{equation*}
\delta ^{1}f(x\otimes y) =\mu(f(x)\otimes y)+ \mu(x \otimes f(y)) -f(\mu(x\otimes y)).
\end{equation*}
In order to simplify the notation, the multiplication is denoted by
concatenation of terms and the tensor product is removed on the right hand side.  Then
\begin{align*}
\delta^{2}(\delta ^{1}f )(x \otimes y \otimes z)
&= \alpha(x)f(yz)+f(\alpha(x))(yz)-f[\alpha(x)(yz)] +\alpha(x)[ yf(z)+f(y)\;z-f(yz) ]+\\
&+\alpha(y)f(xz)+f(\alpha(y))(xz)-f[ \alpha(y)(xz)] +\alpha(y)[ xf(z)+f(x)\;z-f(xz) ] +\\
&-[(xy)f(\alpha(z))+f(xy)\alpha(z)-f((xy)\alpha(z))  ]-  [xf(y)+f(x) \;y - f(xy)]\alpha(z)+ \\
&-[(yx)f(\alpha(z))+f(yx)\;\alpha(z)-f( (yx)\alpha(z)) ] - [ yf(x)+f(y)\;x -f(yx)]\alpha(z),  \\
&=0,
\end{align*}
since $f$ and $\alpha$ commute and the multiplication $\mu$ is Hom-alternative.$\Box$\\
\noindent
The group of the images of $\delta^1$, denoted ${\mathit{B^{2}}(A ,A )}$, corresponds to  the 2-coboundaries and the kernel of $\delta^2$, denoted ${\mathit{Z}^{2}}(A ,A )$, gives the 2-cocycles. The 2-cohomology group is defined as the quotient
$
\mathit{H^{2}}(A ,A )
=\frac{\mathit{Z^{2}}(A
,A)}{\mathit{B^{2}}(A ,A )}.
$

\begin{remark}
To define higher order differentials   one needs to consider  left Hom-alternative multiplicative  $p$-cochain, which are a linear maps
 $f: A^{\otimes p}\rightarrow A$, satisfying $$\alpha \circ f(x_0,...,x_{n-1})=f\big(\alpha (x_0),\alpha(x_1),...,\alpha(x_{n-1})\big) \; \hbox{for all}\; x_0,x_1,...,x_{n-1} \in A.$$
 In \cite{AEM}, it is shown that $\mathcal{C}( A,
A )=\oplus _{p=0}^\infty\mathcal{C}^p ( A,
A )$ carries a structure of Gerstenhaber algebra and leads to a complex for Hom-associative algebras. It turns out that the operad of alternative algebras is not Koszul \cite{DzumadAlternative}.  We conjecture that it is also the case for Hom-alternative algebras.    Obviously one may   set $\delta^{p}=0$ for  $p>3$ but we expect that there is a nontrivial  minimal model.
\end{remark}

\noindent
The deformation equation \eqref{equaDefo} may be written  using  coboundary operators as
\begin{eqnarray}
 \delta^2\mu_k (x\otimes y\otimes z)& = -
\sum_{i=1}^{k-1}{\mu _i} ( \alpha(x) \otimes \mu _{k-i} (
y\otimes z ))-\mu _i ( \mu
_{k-i} ( x\otimes y ) \otimes \alpha(z) ) + \nonumber \\ \
&
 \left.
 \mu _i ( \alpha(y) \otimes \mu _{k-i} (
 x\otimes z ) )
  -\mu _i ( \mu
_{k-i} ( y\otimes x ) \otimes \alpha(z) ), \right.
\end{eqnarray}
where  $k$ is  any nonnegative integer. Hence we have
\begin{proposition}
Let  $(A,\mu_0,\alpha_0)$ be a Hom-alternative algebra and $(A,\mu_t,\alpha_0)$ be a deformation such that $\mu_{t} =\sum_{i\geq 0}\mu_{i}t^{i}.
$ Then $\mu_1$ is a 2-cocycle, that is $\delta^2\mu_1=0$.
\end{proposition}

\noindent
\subsection{Equivalent and trivial deformations}

In this section, we characterize  equivalent as well as trivial
deformations of left Hom-alternative algebras.

\begin{definition}{\rm
Let    $({A },\mu _{0}, \alpha_0)$ be a left Hom-alternative algebra   and let $(A_t,\mu_t, \alpha_t)$ and
$(A'_t,\mu'_t, \alpha '_t)$ be two left Hom-alternative
 deformations of $A$,  where
 $\mu_{t}=\sum_{i\geq 0}t^{i}\mu_{i},\;$  $\;\mu'_{t}=\sum_{i\geq 0}t^{i}\mu'_{i}$,
 with
 $\mu_{0}=\mu'_{0}$, and
 $\alpha_{t}=\sum_{i\geq 0}t^{i}\alpha_{i},\;$  $\; \alpha '_{t}=\sum_{i\geq 0}t^{i}\alpha '_{i}$,
 with
 $\alpha_{0}=\alpha '_{0}$.\\
 We say that the two deformations are \emph{equivalent} if there
exists a formal isomorphism $\rho_{t}: A[[t]]\rightarrow A[[t]]$,
i.e. a $\K[[t]]$-linear map that may be written in the form $
\rho_{t}=\sum_{i\geq 0}t^{i}\rho _{i} ={\rm id}+t\rho
_{1}+t^{2}\rho_{2}+\ldots$, where $\rho_{i}\in End_{\K }(A)$ and
$\rho_{0}=id $ are such that the
following relations hold
\begin{equation}
\label{equIso} \rho_{t}\circ
\mu_{t}=\mu'_{t}\circ(\rho _{t}\otimes \rho _{t}
) \; \;\text{and} \;\; \alpha_t ' \circ  \rho_{t}=  \rho_{t} \circ \alpha_t .
\end{equation}
A deformation $A_{t}$ of $A_{0}$ is said to be \emph{trivial} if
and only if $A_{t}$ is equivalent to $A_{0}$ (viewed  as a
left Hom-alternative algebra  on $A[[t]]$).
}
\end{definition}
\noindent
We discuss in the following the equivalence of two deformations. The two identities in \eqref{equIso} may be written as
\begin{equation}\label{isom1}
\rho _{t}(\mu_{t}(x\otimes y)) =\mu'_{t}(\rho _{t}(x)\otimes \rho
_{t}(y)),\quad\forall x,y\in A .
\end{equation}
and
\begin{equation}\label{isom2}
\rho _{t}(\alpha_{t}(x)) =\alpha '_{t}(\rho _{t}(x)),\quad\forall x\in A .
\end{equation}

\noindent
Equation \eqref{isom1} is equivalent to
\begin{equation}
\sum_{i,j\geq 0}\rho _{i}(\mu_{j}(x\otimes y))t^{i+j} =
\sum_{i,j,k\geq 0}\mu'_{i}(\rho _{j}(x)\otimes
\rho_{k}(y))t^{i+j+k}.
\end{equation}
By identification of  the coefficients, one obtains
that the constant coefficients are identical,
i.e.
\begin{equation*}
\mu_{0}=\mu'_{0} \quad\text{because}\quad \rho
_{0} =id.
\end{equation*}
For  the coefficients of $t$

 and since $\varphi _{0}=id$, it follows that
\begin{equation}
\mu_{1}(x, y)+\rho _{1}(\mu_{0}(x\otimes y)) =
\mu'_{1}(x\otimes y)+\mu_{0}(\rho _{1}(x)\otimes y) +
\mu_{0}(x\otimes \rho _{1}(y)).
\end{equation}
Consequently,
\begin{equation}\label{equiv1}
\mu'_{1}(x\otimes y) = \mu_{1}(x\otimes y)+\rho_{1}(\mu_{0}(x\otimes y))-\mu_{0}(\rho
_{1}(x)\otimes y) - \mu_{0}(x\otimes \rho _{1}(y)).
\end{equation}
\noindent
The homomorphism condition of equation \eqref{isom2} is equivalent to

\begin{equation*}
\sum_{i,j\geq 0}\rho _{i}(\alpha_{j}(x))t^{i+j} =
\sum_{i,j,\geq 0}\alpha '_{i}(\rho _{j}(x))t^{i+j},
\end{equation*}
 which gives $\alpha_0=\alpha '_0$ modulo $t$, and
 \begin{equation} \label{alph}
 \rho_1 \circ \alpha_0+ \rho_0 \circ \alpha_1 = \alpha '_0 \circ \rho_1+ \alpha '_1 \circ \rho_0\;\; \text{modulo} \; t^2.
 \end{equation}

\noindent
The second order conditions of the equivalence between two
deformations of a left Hom-alternative algebra
is given by \eqref{equiv1} which may be written as

\begin{equation}\label{equiv1P}
\mu'_{1}(x\otimes y) = \mu_{1}(x\otimes y)-\delta^1\rho_{1}(x\otimes y) .
\end{equation}

\noindent
In general, if the deformations $\mu _t$ and $\mu _t^{\prime
}$ of $\mu _0$ are equivalent then $\mu _1^{\prime }=\mu _1+\delta^1
f_1$.
Therefore, we have the following

\begin{proposition}{\it Let  $(A,\mu_0,\alpha_0)$ be a Hom-alternative algebra and $(A,\mu_t,\alpha_0)$ be a deformation such that $\mu_{t} =\sum_{i\geq 0}\mu_{i}t^{i}.
$
 The integrability of $\mu _1$ depends only on its cohomology class.}
 \end{proposition}

When we deform only the multiplication,
the elements of $H^2\left( A,A\right) $ give the
infinitesimal deformations ( $\mu _t=\mu _0+t\mu _1).$
We have also  the following

\begin{proposition}{\it
Let $({A },\mu_{0}, \alpha_0)$ be a left Hom-alternative algebra. There is, over $\K [[t]]/t^2,$
a one-to-one correspondence between the
elements of $\mathit{H^{2}}(A  ,A)$ and the
infinitesimal deformation of ${A }$ defined by
\begin{equation}
\mu_{t}(x\otimes y) =\mu_{0}(x\otimes y)+t \mu_{1}(x\otimes y), \quad \forall x,y\in A.
\end{equation}
}
\end{proposition}
\begin{proof}The deformation equation is equivalent to
$\delta^2 \mu_{1}=0$.
\end{proof}

\section{ Deformation by Composition }

In this section, we construct deformations of left alternative algebras using the composition Theorem given in  \cite{Makhlouf10}, which provides a method of obtaining left Hom-alternative
algebras starting from  a left alternative algebra and an algebra endomorphism. Also, we use the notion of $n$th derived Hom-algebras introduced  in \cite{Yau4} to  construct other deformation of left alternative algebras viewed as Hom-algebras. We start by recalling the composition theorems  of left alternative algebras.

\begin{theorem}[\cite{Makhlouf10}]\label{thmConstrHomAlt}
Let $(A,\mu)$ be a left  alternative algebra and  $\alpha : A\rightarrow A$ be an
 algebra endomorphism. Set $\mu_\alpha=\alpha\circ\mu$, then $(A,\mu_\alpha,\alpha)$   is a left  Hom-alternative
algebra.\\
\noindent
Moreover, suppose that  $(A',\mu')$ is another left  alternative algebra  and $\alpha ' : A'\rightarrow A'$ is an algebra
endomorphism. If $f:A\rightarrow A'$ is an algebras morphism that
satisfies $f\circ\alpha=\alpha'\circ f$ then
$$f:(A,\mu_\alpha,\alpha)\longrightarrow (A',\mu'_{\alpha '},\alpha ')
$$
is a morphism of left  Hom-alternative algebras.
\end{theorem}

\begin{remark}\label{HomAltInducedByAlt}Theorem (\ref{thmConstrHomAlt}) gives a procedure of constructing Hom-alternative algebras
using ordinary alternative algebras and their algebra
endomorphisms.
More generally, given a left Hom-alternative algebra $\left( A,\mu ,\alpha \right) $, one may ask whether this
Hom-alternative algebra is   induced by an
ordinary alternative algebra $(A,\widetilde{\mu})$, that is $\alpha$
is an algebra endomorphism with respect to $\widetilde{\mu}$ and
$\mu=\alpha\circ\widetilde{\mu}$. This question was addressed and discussed for
Hom-associative algebras in \cite{FregierGohr2,FregierGohrSilvestrov,Gohr}.
\noindent
Now, if $\alpha$
is an algebra endomorphism with respect to $\widetilde{\mu}$
then $\alpha$ is also an algebra endomorphism
with respect to $\mu$. Indeed,
$$\mu(\alpha(x),\alpha(y))=\alpha\circ\widetilde{\mu}(\alpha(x),\alpha(y))=
\alpha\circ\alpha\circ\widetilde{\mu}(x,y)=\alpha\circ\mu(x,y).$$
\noindent
If $\alpha$ is bijective then $\alpha^{-1}$ is
also an algebra automorphism. Therefore one may use an untwist
operation on the Hom-alternative algebra in order to recover the
alternative algebra ($\widetilde{\mu}=\alpha^{-1}\circ\mu$).
\end{remark}
\noindent
The previous procedure was generalized in \cite{Yau4} to $n$th derived Hom-algebras. We split the  definition given in  \cite{Yau4} into two types of   $n$th derived Hom-algebras.
\begin{definition}[\cite{Yau4}]\label{defiNDerivedHom} Let  $\left( A,\mu ,\alpha \right) $ be a multiplicative Hom-algebra and $n\geq 0$. The  $n$th derived Hom-algebra of type $1$ of $A$ is  defined by
\begin{equation}\label{DerivedHomAlgtype1}
A^n=\left( A,\mu^{(n)}=\alpha^{n }\circ\mu ,\alpha^{n+1} \right),
\end{equation}
and the $n$th derived Hom-algebra of type $2$ of $A$ is  defined by
\begin{equation}\label{DerivedHomAlgtype2}
A^n=\left( A,\mu^{(n)}=\alpha^{2^n -1}\circ\mu ,\alpha^{2^n} \right).
\end{equation}
Note that in both cases  $A^0=A$, $A^1=\left( A,\mu^{(1)}=\alpha\circ\mu ,\alpha^{2} \right)$ and $A^{n+1}=(A^n)^1.$
\end{definition}
Observe that for $n\geq 1$ and $x,y,z\in A$ we have
\begin{eqnarray*}
\mu^{(n)}(\mu^{(n)}(x,y),\alpha^{n+1}(z))&=& \alpha^{n }\circ\mu(\alpha^{n }\circ\mu(x,y),\alpha^{n+1}(z))\\
\ &=& \alpha^{2n }\circ\mu(\mu(x,y),\alpha(z)).
\end{eqnarray*}
Therefore, following  \cite{Yau4}, one obtains the following result.
\begin{theorem}\label{ThmConstrNthDerivedAlternative}
Let   $\left( A,\mu ,\alpha \right) $ be a multiplicative left Hom-alternative algebra (resp. Hom-alternative algebra). Then the $n$th derived Hom-algebra of type $1$ is  also a left Hom-alternative algebra (resp. Hom-alternative algebra).
\noindent
The same holds for multiplicative Hom-associative algebras.
\end{theorem}

\noindent
Now, we use  Theorem \ref{thmConstrHomAlt} and Theorem  \ref{ThmConstrNthDerivedAlternative} to obtain a procedure of deforming left Hom-alternative algebras  by composition.

\begin{proposition}\label{DeformComposition}
Let $(A,\mu_0)$ be a left alternative algebra and $\alpha_t$ be an algebra endomorphism of the form
$\alpha_{t}=id+ \sum_{i\geq 1}^{p}t^{i}\alpha_{i}$, where $\alpha_i$ are linear maps on $A$, $t$ is a parameter in $\K$  and $p$ is an integer. Set  $\mu_t=\alpha_t\circ\mu$.
\\ Then $(A,\mu_t,\alpha_t)$  is a left  Hom-alternative
algebra which is a deformation of  the alternative algebra viewed as a Hom-alternative algebra $(A,\mu_0, id)$.

Moreover, the $n$th derived Hom-algebra of type $1$
\begin{equation}
A_t^n=\left( A,\mu_t^{(n)}=\alpha_t^{n }\circ\mu_t ,\alpha_t^{n+1} \right),
\end{equation}
is a deformation of  $(A,\mu, id)$.
\end{proposition}
\begin{proof}
The first assertion follows from  Theorem \ref{thmConstrHomAlt}. In particular for an infinitesimal deformation of the identity $\alpha_{t}=id+t\alpha_{1}$, we have $\mu_t=\mu +t \alpha_1\circ\mu$.

The proof of the left Hom-alternativity of   the  $n$th derived Hom-algebra $(A,\mu_t,\alpha_t)$  follows from the fact  that, for $n\geq 1$ and $x,y,z\in A$, we have
\begin{eqnarray*}
\mu_t^{(n)}(\mu_t^{(n)}(x,y),\alpha_t^{n+1}(z))&=& \alpha_t^{n }\circ\mu_t( \alpha_t^{n }\circ\mu_t(x,y),\alpha_t^{n+1}(z))\\
\ &=& \alpha_t^{2n }\circ\mu_t(\mu_t(x,y),\alpha_t(z)).
\end{eqnarray*}
In case $n=1$ and $\alpha_{t}=id+t\alpha_{1}$ the multiplication is of the form
\begin{eqnarray*}
\mu_t^{(1)}&=& (id+t\alpha_{1})\circ(id+t\alpha_{1})\circ\mu\\
\ &=&\mu+2 t \alpha_{1}\circ \mu+t^2 \alpha_{1}^2\circ \mu
\end{eqnarray*}
and the twist is $\alpha^2_t=(id+t \alpha_1)^2=id +2 t \alpha_1+t^2  \alpha^2_1 .$ Therefore we get another deformation of the alternative algebra viewed as a Hom-alternative algebra $(A,\mu_0, id)$. The proof in the general case is similar.

\end{proof}
\begin{remark}
Proposition \ref{DeformComposition} is valid for Hom-associative algebras, $G$-Hom-associative algebras and Hom-Lie algebras.
\end{remark}

\begin{remark}
More generally, if $(A,\mu,\alpha)$ is  a multiplicative left Hom-alternative algebra where $\alpha$ may be written of the form $\alpha=id+t\alpha_{1}$ , then  the $n$th derived Hom-algebra of type $1$
\begin{equation}
A_t^n=\left( A,\mu^{(n)}=\alpha^{n }\circ\mu,\alpha^{n+1} \right),
\end{equation}
gives a one parameter formal deformation of  $(A,\mu, \alpha)$. But for any $\alpha$ one obtains just  new left Hom-alternative algebras.
\end{remark}

\section{Examples of  Deformations and Computations  }
In this section we provide examples of deformations of left alternative algebras. Using Proposition  \ref{DeformComposition}
 we construct  left Hom-alternative formal deformations of the $4$-dimensional left alternative algebras which are not associative (see \cite{EGG}). These algebras are viewed as left Hom-alternative algebras with identity map as a twist.
To this end,   for each algebra, we provide all the algebra endomorphisms which are infinitesimal deformations of the identity, that is of the form $\alpha_t=id+t \alpha_1$, where $\alpha_1$ is a linear map. Therefore, left Hom-alternative
 algebras are obtained from left alternative algebras and correspond to left Hom-alternative formal deformations of these left alternative algebras.

There are exactly two alternative but not associative algebras of
dimension 4 over any field \cite{EGG}.
With respect to a basis $\{e_0, e_1, e_2, e_3\}$, one algebra is given by the
following multiplication
\begin{eqnarray}\label{dim4-1} && \mu_{4,1}(e_0,e_0)=e_0, \; \mu_{4,1} (e_0,e_1)=e_1,
\;\mu_{4,1} (e_2,e_0)=e_2,\\&&
\mu_{4,1} (e_2,e_3)=e_1, \;\mu_{4,1} (e_3,e_0)=e_3, \;\mu_{4,1} (e_3,e_2)=- e_1.\nonumber
\end{eqnarray}
The other algebra is given by
\begin{eqnarray}\label{dim4-2}&& \mu_{4,2} ( e_0,e_0)=e_0, \;\mu_{4,2} ( e_0,e_2)=e_2,
 \;\mu_{4,2} (e_0,e_3)=e_3,
 \\&& \mu_{4,2} (e_1,e_0)=e_1, \;\mu_{4,2} (e_2,e_3)=e_1, \;\mu_{4,2} (e_3,e_2)=- e_1.\nonumber
 \end{eqnarray}
These two alternative algebras are anti-isomorphic, that is the first
one is isomorphic to the opposite of the second one.

In the following we characterize  the homomorphisms which induce  left Hom-alternative structures on the   $4$-dimensional left alternative algebras, given by \eqref{dim4-1} and \eqref{dim4-2}  which are not associative and also on  the   Octonions algebra (see  \cite{Baez} or \cite{Makhlouf10} for the multiplication table). Straightforward calculations give

\begin{proposition}
The only   left Hom-alternative algebra structure on   $4$-dimensional left alternative algebras  defined by  the
 multiplications \eqref{dim4-1} or \eqref{dim4-2}
 is given by the identity homomorphism. The same holds for the octonions algebra.
\end{proposition}

\subsection{Examples of Deformations by Composition}
 We provide  for  the $4$-dimensional left alternative algebras  defined by  the
 multiplications \eqref{dim4-1} or \eqref{dim4-2}   the algebra endomorphisms which may be viewed as infinitesimal  deformations of the identity. Then according to Theorem \eqref{thmConstrHomAlt}, the linear maps yield  Hom-alternative  infinitesimal deformations of the  left alternative algebras defined by $\mu_{4,1}$ and $\mu_{4,2}$.

\begin{proposition}\label{famille1_mu1}
 The infinitesimal deformations  $\alpha$  of the identity which are algebra endomorphisms of the alternative algebra  $\mu_{4,1}$, defined above by the equation  \eqref{dim4-1},  are  given with respect to the same basis  by
\begin{align*}
 \alpha(e_0)&= e_0+t(a_1\ e_1+a_2\ e_2+a_3\ e_3 ),\\
 \alpha(e_1)&=e_1+t(-a_4 a_5+a_6+ a_7+a_6 a_7)\ e_1,\\
\alpha(e_2)&=e_2+t[(a_3 -a_2 a_5+a_3 a_6)e_1 + a_6\ e_2+a_5 \ e_3],\\
 \alpha(e_3)&=e_3+t[(-a_2+ a_3 a_4-a_2 a_7)e_1 + a_4\ e_2+a_7 \ e_3].
\end{align*}
where $a_1,\cdots,a_7\in \K$ are free parameters.

Hence, the linear map $\alpha$ and the following multiplication $\widetilde{\mu_{4,1}}$ defined by
 \begin{align*}
  \widetilde{\mu_{4,1}}(e_0,e_0)&=e_0+t(a_1\ e_1+a_2\ e_2+a_3\ e_3 ), \\
\widetilde{\mu_{4,1}} (e_0,e_1)&=e_1+t(-a_4 a_5+a_6+ a_7+a_6 a_7)\ e_1,\\
\widetilde{\mu_{4,1}} (e_2,e_0)&=e_2+t[(a_3 -a_2 a_5+a_3 a_6)e_1 + a_6\ e_2+a_5 \ e_3],\\
\widetilde{\mu_{4,1}} (e_2,e_3)&=e_1+t(-a_4 a_5+a_6+ a_7+a_6 a_7)\ e_1, \\
\widetilde{\mu_{4,1}} (e_3,e_0)&=e_3+t[(-a_2+ a_3 a_4-a_2 a_7)e_1 + a_4\ e_2+a_7 \ e_3],\\
\widetilde{\mu_{4,1}} (e_3,e_2)&=-e_1-t(-a_4 a_5+a_6+ a_7+a_6 a_7)\ e_1.
\end{align*}
determine  4-dimensional Hom-alternative algebras.

\end{proposition}

For the alternative algebra defined by $\mu_{4,2}$ we obtain
\begin{proposition}
 The infinitesimal deformations  $\alpha$  of the identity which are algebra endomorphisms of the alternative algebra  $\mu_{4,2}$, defined above by the equation  \eqref{dim4-2},  are  given with respect to the same basis  by
\begin{align*}
 \alpha(e_0)&= e_0+t(a_1\ e_1+a_2\ e_2+a_3\ e_3 ),\\
 \alpha(e_1)&=e_1+t(-a_4 a_5+a_6+ a_7+a_6 a_7)\ e_1,\\
\alpha(e_2)&=e_2+t[-(a_3 -a_2 a_5+a_3 a_6)e_1 + a_6\ e_2+a_5 \ e_3],\\
 \alpha(e_3)&=e_3+t[-(-a_2+ a_3 a_4-a_2 a_7)e_1 + a_4\ e_2+a_7 \ e_3].
\end{align*}
where $a_1,\cdots,a_7\in \K$ are free parameters.

Hence, the linear map $\alpha$ and the multiplication  $\widetilde{\mu_{4,2}}$ defined by
 \begin{align*}
  \widetilde{\mu_{4,2}}(e_0,e_0)&= e_0+t(a_1\ e_1+a_2\ e_2+a_3\ e_3 ), \\
\widetilde{\mu_{4,2}} (e_0,e_2)&=e_2+t[-(a_3 -a_2 a_5+a_3 a_6)e_1 + a_6\ e_2+a_5 \ e_3],\\
\widetilde{\mu_{4,2}} (e_0,e_3)&=e_3+t[-(-a_2+ a_3 a_4-a_2 a_7)e_1 + a_4\ e_2+a_7 \ e_3],\\
\widetilde{\mu_{4,2}} (e_1,e_0)&=e_1+t(-a_4 a_5+a_6+ a_7+a_6 a_7)\ e_1, \\
\widetilde{\mu_{4,2}} (e_2,e_3)&=e_1+t(-a_4 a_5+a_6+ a_7+a_6 a_7)\ e_1,\\
\widetilde{\mu_{4,2}} (e_3,e_2)&=-e_1-t(-a_4 a_5+a_6+ a_7+a_6 a_7)\ e_1.
\end{align*}
determine  4-dimensional Hom-alternative algebras.

\end{proposition}

\begin{remark}
We may use Proposition \ref{DeformComposition} to construct deformations of higher degrees. For example
$$\mu_t^{(1)}=\mu+2 t \alpha\circ \mu+t^2 \alpha^2\circ \mu
$$and
$$\alpha_t=id +2 t \alpha+t^2  \alpha^2$$
define deformations of degree $3$ by setting $\mu=\mu_{4,1}$ (resp. $\mu=\mu_{4,2}$).
\end{remark}

\subsection{Derivations and cocycles}
In the following we compute the derivations and the 2-cocycles of some  $4$-dimensional alternative algebras.

Recall that a derivation of a left Hom-alternative algebra $(A,\mu,\alpha)$ is given by a linear map $f:A\rightarrow A$ satisfying
$$\mu(f(x)\otimes y)+ \mu(x \otimes f(y)) -f(\mu(x\otimes y))=0.
$$
\begin{proposition}
The derivations of the $4$-dimensional alternative algebras defined by   \eqref{dim4-1} and  \eqref{dim4-2} viewed as Hom-alternative algebras are given with respect to the same basis by
\begin{align*}
f(e_0)&=b_1\ e_1+b_2\ e_2+b_3\ e_3 ,\quad &
  f(e_1)&=(b_4+ b_5)\ e_1,\\
f(e_2)&=b_3 e_1 + b_4\ e_2+b_6 \ e_3,\quad &
f(e_3)&=-b_2 e_1 + b_7\ e_2+b_5 \ e_3.
\end{align*}
where  $b_1,\cdots,b_7\in \K$ are free parameters.
\end{proposition}
Recall that a 2-cocycle  of a left Hom-alternative algebra $(A,\mu,\alpha)$ is given by a linear map $\varphi:A\times A\rightarrow A$ satisfying
\begin{eqnarray}\label{delta2bis2}
\delta^2\varphi(x,y,z)&=\mu\circ(\varphi(x,y), {\alpha}(z) )-\mu( {\alpha}(x) , \varphi(y,z))+\varphi\circ(\mu(x,y), {\alpha}(z))-\varphi({\alpha}(x), \mu(y,z))\\
\ &+\mu\circ(\varphi(y,x), {\alpha}(z) )-\mu( {\alpha}(y) , \varphi(x,z))+\varphi\circ(\mu(y,x), {\alpha}(z))-\varphi({\alpha}(y), \mu(x,z))\nonumber\\
\ & =0. \nonumber
\end{eqnarray}
We then have the following
\begin{proposition}
The  2-cocycles of the $4$-dimensional alternative algebras defined in   \eqref{dim4-1}  viewed as Hom-alternative algebras are given with respect to the same basis by
\begin{align*}
\varphi(e_0,e_0)&=\lambda_1\ e_0 ,\quad &
 \varphi(e_0,e_1)&=\lambda_1\ e_1 +\lambda_2 e_2+\lambda_3\ e_3 ,\\
\varphi(e_0,e_2)&=\lambda_4\ e_0-\lambda_5\ e_1  ,\quad &
\varphi(e_0,e_3)&=\lambda_6\ e_0-\lambda_7\ e_1,\\
\varphi(e_1,e_0)&=\lambda_8\ e_0- \lambda_2\ e_2-\lambda_3\ e_3,\quad &
 \varphi(e_1,e_1)&=\lambda_8\ e_1 ,\\
\varphi(e_1,e_2)&=\lambda_3\ e_1 ,\quad &
\varphi(e_1,e_3)&=-\lambda_2\ e_1,\\
\varphi(e_2,e_0)&=\lambda_9\ e_1 +\lambda_1\ e_2 ,\quad &
 \varphi(e_2,e_1)&=(\lambda_{4}- \lambda_3)\ e_1 +\lambda_8 e_2 ,\\
\varphi(e_2,e_2)&=\lambda_4\ e_2  ,\quad &
\varphi(e_2,e_3)&=-\lambda_8\ e_0-\lambda_{10}\ e_1+(\lambda_{2}+ \lambda_6)\ e_2+\lambda_3\ e_3,\\
\varphi(e_3,e_0)&=\lambda_7\ e_1+ \lambda_1\ e_3,\quad &
 \varphi(e_3,e_1)&=(\lambda_{2}+ \lambda_6)\ e_1+\lambda_8\ e_3  ,\\
\varphi(e_3,e_2)&=\lambda_8\ e_0+\lambda_{10}\ e_1 -\lambda_2\ e_2+(\lambda_{4}- \lambda_3)\ e_3   ,\quad &
\varphi(e_3,e_3)&=\lambda_6\ e_3  .
\end{align*}
where  $\lambda_i \in \K$ with $1\leq i\leq 10$ are free parameters.

The cohomology classes of these 2-cocycles are trivial. Hence the Hom-alternative algebra is rigid in the class of alternative algebras.
\end{proposition}
For the  $4$-dimensional alternative algebras defined in   \eqref{dim4-2} we obtain

\begin{proposition}
The  2-cocycles of the $4$-dimensional alternative algebras defined in   \eqref{dim4-2}  viewed as Hom-alternative algebras are given with respect to the same basis by
\begin{align*}
\varphi(e_0,e_0)&=\lambda_1\ e_0 ,\quad &
 \varphi(e_0,e_1)&=\lambda_2\ e_0 +\lambda_3 e_2+\lambda_4\ e_3 ,\\
\varphi(e_0,e_2)&=-\lambda_5\ e_1-\lambda_1\ e_2  ,\quad &
\varphi(e_0,e_3)&=\lambda_6\ e_1-\lambda_1\ e_3,\\
\varphi(e_1,e_0)&=\lambda_1\ e_1- \lambda_3\ e_2-\lambda_4\ e_3,\quad &
 \varphi(e_1,e_1)&=\lambda_2\ e_1 ,\\
\varphi(e_1,e_2)&=(\lambda_7-\lambda_8 )\ e_1+\lambda_2\ e_2 ,\quad &
\varphi(e_1,e_3)&=(\lambda_9+\lambda_{10} )\ e_1+\lambda_2\ e_3,\\
\varphi(e_2,e_0)&=\lambda_7\ e_0 +\lambda_5\ e_1 ,\quad &
 \varphi(e_2,e_1)&=\lambda_{8}\ e_1,\\
\varphi(e_2,e_2)&=\lambda_7\ e_2  ,\quad &
\varphi(e_2,e_3)&=-\lambda_2\ e_0+\lambda_{11}\ e_1-\lambda_{3}\ e_2+(\lambda_7-\lambda_8 )\ e_3,\\
\varphi(e_3,e_0)&=\lambda_9\ e_0- \lambda_6\ e_1,\quad &
 \varphi(e_3,e_1)&=-\lambda_{10}\ e_1,\\
\varphi(e_3,e_2)&=\lambda_2\ e_0-\lambda_{11}\ e_1+(\lambda_9+\lambda_{10} )\ e_2+\lambda_{8}\ e_3   ,\quad &
\varphi(e_3,e_3)&=\lambda_9\ e_3  .
\end{align*}
where  $\lambda_i \in \K$ with $1\leq i\leq 11$ are free parameters.

The cohomology classes of these 2-cocycles are trivial. Hence the Hom-alternative algebra is rigid in the class of alternative algebras.
\end{proposition}

We compute now the derivation, the 2-cocycles and the 2-cohomology group of the 4-dimensional Hom-alternative algebras, which are not alternative algebras,  given by
 \begin{align}\label{Dim4famMu}
  \mu(e_0,e_0)&=e_0+t a_1\ e_1 , \quad &
\mu (e_0,e_1)&=e_1+t(-a_2 a_3+a_4)\ e_1,\\
 \mu (e_2,e_0)&=e_2+t  a_4\ e_2+t a_3 \ e_3,\quad &
 \mu (e_2,e_3)&=e_1+t(-a_2 a_3+a_4)\ e_1, \nonumber\\
 \mu (e_3,e_0)&=e_3+t  a_2\ e_2 ,\quad &
 \mu (e_3,e_2)&=-e_1-t(-a_2 a_3+a_4)\ e_1.\nonumber
\end{align}
\begin{align}\label{Dim4famAlfa}
 \alpha(e_0)&= e_0+t a_1\ e_1 ,\quad
 & \alpha(e_1)&=e_1+t(-a_2 a_3+a_4)\ e_1,\\
\alpha(e_2)&=e_2+t a_4\ e_2+t a_3 \ e_3,\quad
& \alpha(e_3)&=e_3+t  a_2\ e_2.\nonumber
\end{align}
where $a_1,a_2,a_3,a_4\in \K$ are free parameters. These Hom-alternative algebras are particular cases of the Hom-alternative algebras constructed in Proposition \ref{famille1_mu1}.
\begin{proposition}
The derivations $f$  of the $4$-dimensional Hom-alternative algebras defined in   \eqref{Dim4famMu} and \eqref{Dim4famAlfa}  are given with respect to the same basis by
\begin{align*}
f(e_0)&=b_1\ e_1 ,\\
 f(e_1)&=-b_1(\frac{a_2 a_3-a_4}{a_1})e_1,\\
f(e_2)&=- \frac{b_1a_2 a_3-b_1a_4+b_2a_1}{a_1}\ e_2- \frac{a_3(b_1a_2 a_3-b_1a_4+2 b_2a_1)}{a_1a_4}\ e_3,\\
f(e_3)&=- \frac{a_2(b_1a_2 a_3-b_1a_4+2b_2a_1)}{a_1a_4}\ e_2+b_2\ e_3.
\end{align*}
where  $b_1,b_2 \in \K$ are free parameters.
\end{proposition}
and for the 2-cocycles we obtain
\begin{proposition}
The only 2-cocycle of the $4$-dimensional Hom-alternative algebras defined in   \eqref{Dim4famMu} and \eqref{Dim4famAlfa} corresponds to the multiplication  \eqref{Dim4famMu}. Moreover there exist a linear map $g$ such that $\mu=\delta^1 g$ with $g$ defined by
\begin{align*}
g(e_0)&=e_0, \\
 g(e_1)&=\nu_1e_0 +(\nu_2+\nu_3-1)e_1+\nu_4 e_2 -\nu_5 e_3,\\
g(e_2)&=\nu_5 e_0 + \nu_2e_2+\nu_6e_3,\\
g(e_3)&=\nu_4 e_0 + \nu_7e_2+\nu_8e_3.
\end{align*}
where  $\nu_i \in \K$ with $1\leq i\leq 8$ are free parameters.

Hence the  cohomology class of $\mu$ is zero.
\end{proposition}

\section{Deformations of Hom-Malcev algebras}
In this section we study the formal deformation of Hom-Malcev algebras. Under the same assumptions as in Section 2 we define the 1-parameter formal deformation of a Malcev algebra as
\begin{definition}

Let $A$ be a $\K$-vector space and  $(A, [\ , \ ]_0, \alpha_0)$ be
a Hom-Malcev algebra. A \emph{ Hom-Malcev 1-parameter formal deformation}
 is given by the $\K[[t]]$-bilinear map $ [\ , \ ]_{t}:
A[[t]] \times A[[t]]  \rightarrow A[[t]]$ and $\K[[t]]$-linear map
$\alpha_{t}: A[[t]] \rightarrow  A[[t]] $  of the form
\begin{equation} \label{def:deformHomMalcev}
[\ , \ ]_{t} =\sum_{i\geq 0}[\ , \ ]_{i}t^{i}
\quad\text{ and }\quad \alpha_{t} = \sum_{i\geq
0}\alpha_{i}t^{i}
\end{equation}
where each $[\ , \ ]_{i}:  A\times A  \rightarrow A$
is a $\K$-bilinear map (extended to be
$\K[[t]]$-bilinear) and each $\alpha_{i}:  A
\rightarrow A$ is a $\K$-linear map (extended to
be $\K[[t]]$-linear),  satisfying for $x, y,z\in
 A$ the following identity:
 \begin{equation}\label{HomMalcevDEFequation} J_{\alpha_t}(\alpha_t(x),\alpha_t(y),[x,z]_t)=[J_{\alpha_t}(x,y,z),\alpha_t^2(x)]_t
\end{equation}
 where $J_{\alpha_t}$ is the Hom-Jacobiator  defined for $x,y,z\in A$ by $J_{\alpha_t}(x,y,z)=\circlearrowleft_{x,y,z}{[[x,y]_t,\alpha_t(z)]_t}.$
\end{definition}

We have
\begin{eqnarray*}
J_{\alpha_t}(x,y,z)&=&\circlearrowleft_{x,y,z}{[[x,y]_t,\alpha_t(z),]_t}\\
\ &=& \sum_{i,j,k\geq0}{\circlearrowleft_{x,y,z}{[[x,y]_i,\alpha_k(z)]_j t^{i+j+k}}}.
\end{eqnarray*}

We introduce the following notation $J_{\alpha}^{i,j}$ which is a trilinear map defined by
\begin{equation}
J_{\alpha}^{i,j}(x,y,z)=\circlearrowleft_{x,y,z}{[[x,y]_i,\alpha (z)]_j},
\end{equation}
where $\alpha$ is a linear map and  $[\ ,\ ]_i$ and $[\ ,\ ]_j$ are bilinear maps.

Therefore the left hand side of the identity \eqref{HomMalcevDEFequation} gives
\begin{equation}\label{lhsDE}
J_{\alpha_t}(\alpha_t(x),\alpha_t(y),[x,z]_t)=
\sum_{i,j,k,p,q,r\geq 0}{J_{\alpha_r}^{i,j}(\alpha_p(x),\alpha_q(y),[x,z]_k )t^{i+j+k+p+q+r}}.
\end{equation}
While the right hand side of \eqref{HomMalcevDEFequation} gives
\begin{eqnarray}\label{rhsDE}
[J_{\alpha_t}(x,y,z),\alpha_t^2(x)]_t&=& \sum_{i,j,k,p,q,r\geq 0}{[(\circlearrowleft_{x,y,z}{[[x,y]_i,\alpha_r (z)]_j}),\alpha_p\circ\alpha_q (x)]_kt^{i+j+k+p+q+r}}\\
\ &=& \sum_{i,j,k,p,q,r\geq 0}{[J_{\alpha_r}^{i,j}(x,y,z),\alpha_p\circ\alpha_q (x)]_kt^{i+j+k+p+q+r}}\nonumber
\end{eqnarray}
Equating \eqref{lhsDE} and  \eqref{rhsDE} yields Hom-Malcev identity of the original Hom-Malcev algebra, for the degree 0 terms. The terms of degree 1 lead to a  12 terms identity which reduces, when the twist map $\alpha_0$ is not deformed to  the following identity
\begin{eqnarray}
& J_{\alpha_0}^{1,0}(\alpha_0(x),\alpha_0(y),[x,z]_0 )+
J_{\alpha_0}^{0,1}(\alpha_0(x),\alpha_0(y),[x,z]_0) +
J_{\alpha_0}^{0,0}(\alpha_0(x),\alpha_0(y),[x,z]_1) \\&
-[J_{\alpha_0}^{1,0}(x,y,z),\alpha_0^2 (x)]_0
-[J_{\alpha_0}^{0,1}(x,y,z),\alpha_0^2 (x)]_0
-[J_{\alpha_0}^{0,0}(x,y,z),\alpha_0^2 (x)]_1=0\nonumber
\end{eqnarray}

This identity and the study of trivial and equivalent deformation suggests to introduce the following 1-coboundary and 2-coboundary operators for a Hom-Malcev algebra $(A,[\ ,\ ],\alpha)$.   Let $ \mathcal{C}_s^n ( A,
A )$ be the set of skewsymmetric $\alpha$-multiplicative $n$-linear maps on $A$.  
We  define the first differential $\delta ^1f \in  \mathcal{C}_s^2 ( A,A )$ by
\begin{equation}
\delta ^1f =[\ ,\ ] \circ \left(
f\otimes  {\rm id} \right) +[\ ,\ ] \circ \left( {\rm id}  \otimes f \right)
-f\circ [\ ,\ ].
\end{equation}

The second differential $\delta^2 \phi  \in \mathcal{C}_s^3 ( A,A )$  where $\phi\in\mathcal{C}_s^2 ( A,A)$ and denoted by $\phi=[\ ,\ ]_1$, is defined by
\begin{eqnarray}\label{delta2Malcev}
&\delta^2\phi(x,y,z)= J_{\alpha_0}^{1,0}(\alpha_0(x),\alpha_0(y),[x,z]_0 )+
J_{\alpha_0}^{0,1}(\alpha_0(x),\alpha_0(y),[x,z]_0) +
J_{\alpha_0}^{0,0}(\alpha_0(x),\alpha_0(y),[x,z]_1)\nonumber \\&
-[J_{\alpha_0}^{1,0}(x,y,z),\alpha_0^2 (x)]_0
-[J_{\alpha_0}^{0,1}(x,y,z),\alpha_0^2 (x)]_0
-[J_{\alpha_0}^{0,0}(x,y,z),\alpha_0^2 (x)]_1
\end{eqnarray}


\subsection{Deformation by composition}
In the sequel we give a procedure of  deforming  Malcev algebras into Hom-Malcev algebras using the following two Theorems.
\begin{theorem}[\cite{Yau4}]\label{MalcevThm1}
Let   $\left( A,[\ ,\ ]  \right) $ be a Malcev algebra and $\alpha $ be an algebra endomorphism on $A$. Then the  Hom-algebra $\left( A,\alpha\circ [\ ,\ ] ,\alpha \right) $ induced by $\alpha$
is  a Hom-Malcev algebra.
\end{theorem}
\begin{theorem}[\cite{Yau4}]\label{MalcevThm2}
Let   $\left( A,[\ ,\ ] ,\alpha \right) $ be a Hom-Malcev algebra. Then the $n$th derived Hom-algebra of type 2
$$\left( A,[\ ,\ ]^{(n)}=\alpha^{2^n -1}\circ[\ ,\ ] ,\alpha^{2^n} \right)$$
is also a Hom-Malcev algebra.
\end{theorem}

\noindent
We provide a procedure to deform Malcev algebras by composition.

\begin{proposition}\label{DeformCompositionMalcev}
Let $(A,[\ ,\ ])$ be a Malcev algebra and $\alpha_t$ be an algebra endomorphism of the form
$\alpha_{t}=id+ \sum_{i\geq 1}^{p}t^{i}\alpha_{i}$, where $\alpha_i$ are linear maps on $A$, $t$ is a parameter in $\K$  and $p$ is an integer. \\ Let  $[\ ,\ ]_t=\alpha_t\circ[\ ,\ ]$, then $(A,[\ ,\ ]_t,\alpha_t)$  is a Malcev
algebra which is a deformation of  the Malcev  algebra viewed as a Hom-Malcev  algebra $(A,[\ ,\ ], id)$.

Moreover, the $n$th derived Hom-algebra of type $2$
\begin{equation}
A_t^n=\left( A,[\ ,\ ]_t^{(n)}=\alpha_t^{2^n -1}\circ[\ ,\ ]_t ,\alpha_t^{2^n} \right),
\end{equation}
is a deformation of  $(A,[\ ,\ ], id)$.
\end{proposition}
\begin{proof}
The first assertion follows from  Theorem \ref{MalcevThm1}. In particular for an infinitesimal deformation of the identity $\alpha_{t}=id+t\alpha_{1}$, we have $[\ ,\ ]_t=[\ ,\ ] +t \alpha_1\circ[\ ,\ ]$.

The proof of the  Hom-Malcev identity of   the  $n$th derived Hom-algebra $(A,[\ ,\ ]_t,\alpha_t)$  follows from the Theorem \ref{MalcevThm2}.
In case $n=1$ and $\alpha_{t}=id+t\alpha_{1}$ the bracket is
\begin{eqnarray*}
[\ ,\ ]_t^{(1)}&=& (id+t\alpha_{1})\circ(id+t\alpha_{1})\circ[\ ,\ ]\\
\ &=&[\ ,\ ]+2 t \alpha_{1}\circ [\ ,\ ]+t^2 \alpha_{1}^2\circ [\ ,\ ]
\end{eqnarray*}
and the twist  map is $\alpha^2_t=(id+t \alpha_1)^2=id +2 t \alpha_1+t^2  \alpha^2_1 .$ Therefore we get another deformation of the Malcev  algebra viewed as a Hom-Malcev algebra $(A,[\ ,\ ], id)$. The proof in the general case is similar.

\end{proof}
\begin{remark}
More generally, if $(A,[\ ,\ ],\alpha)$ is  a multiplicative  Hom-Malcev algebra where $\alpha$ may be written of the form $\alpha=id+t\alpha_{1}$ , then  the $n$th derived Hom-algebra of type $1$
\begin{equation}
A_t^n=\left( A,\mu^{(n)}=\alpha^{n }\circ[\ ,\ ],\alpha^{n+1} \right),
\end{equation}
gives a one parameter formal deformation of  $(A,[\ ,\ ], \alpha)$. But for any $\alpha$ one obtains just  new Hom-Malcev  algebras.
\end{remark}
\begin{example}The example \ref{MalcevExample} is obtained in \cite{Yau4} by computing the algebra endomorphism of the Malcev algebra defined with respect to $\{e_0,e_1,e_2,e_3\} $ by
\begin{equation}\label{MalcevEx}
[e_0,e_1]=-e_1,\quad [e_0,e_2]=-e_2,\quad [e_0,e_3]=e_3,\quad  [e_1,e_2]=2e_3.
\end{equation}
By specializing the parameters, we consider the following class of the algebra endomorphism which are deformation of the identity map
\begin{align}\label{alphaMalcev}
&\alpha (e_0)=e_0+t(e_2+e_3),\quad &&  \alpha (e_1)=e_1+t(e_1+e_2+e_3)+t^2 e_3,\\
&\alpha (e_2)=e_2+te_2,\quad && \alpha (e_3)=e_3+2te_3+t^2 e_3.\nonumber
\end{align}
By  Theorem \ref{DeformCompositionMalcev} we obtain the Hom-Malcev deformation of  the Malcev algebra \eqref{MalcevEx} given by
\begin{align*}
&[e_0,e_1]_t=-e_1-t(e_1+e_2+e_3)-t^2 e_3,
\quad &&[e_0,e_2]_t=-e_2-te_2,\\
& [e_0,e_3]_t=e_3+2te_3+t^2 e_3,
 \quad && [e_1,e_2]_t=2e_3+4te_3+2t^2 e_3.
\end{align*} and the linear map defined in \eqref{alphaMalcev}.

By using once again Theorem \ref{DeformCompositionMalcev} we obtain the following Hom-Malcev deformations defined by
\begin{align*}
&[e_0,e_1]^1_t=-e_1-2 t(e_1+e_2+e_3)-t^2 (e_1+2e_2+5e_3)-4t^3 e_3-t^4 e_3,\\
&[e_0,e_2]^1_t=-e_2-2te_2-t^2 e_2,\\
&[e_0,e_3]^t_t=e_3+4te_3+6t^2 e_3+4t^3e_3+t^4 e_3,\\
& [e_1,e_2]^1_t=2e_3+8te_3+12t^2 e_3+8t^3e_3+2t^4 e_3.
\end{align*}
and
\begin{align*}\label{alphaMalcev2}
&\alpha^2 (e_0)=e_0+t(2e_2+2e_3)+t^2(e_2+2e_3)+t^3 e_3,\\
&  \alpha^2 (e_1)=e_1+2 t(e_1+e_2+e_3)+t^2 (e_1+2e_2+5e_3)+4t^3 e_3+t^4 e_3,\\
&\alpha^2 (e_2)=e_2+2te_2+t^2 e_2,\\
& \alpha^2 (e_3)=e_3+4te_3+6t^2 e_3+4t^3e_3+t^4 e_3.
\end{align*}
\end{example}

\noindent
{\bf Acknowledgment} The second author  would like to thank the Mathematics Department at the University of South Florida for its hospitality while this paper was being finished.

\noindent
\textsc{Department of Mathematics, \\
University of South Florida, \\
4202 E Fowler Ave., \\
Tampa, FL 33620, USA}\\\\
and\\\\
\textsc{Laboratoire de Math\'ematiques, \\
Informatique et Applications, \\
Universit\'e de Haute Alsace, \\
France.}

\end{document}